\documentclass[envcountsect]{svmult}
\usepackage[T1]{fontenc}
\usepackage[latin9]{inputenc}
\usepackage{prettyref}
\usepackage{url}
\usepackage{amsmath}
\usepackage{amssymb}
\usepackage{esint}
\usepackage[unicode=true,
 bookmarks=true,bookmarksnumbered=false,bookmarksopen=false,
 breaklinks=false,pdfborder={0 0 1},backref=false,colorlinks=false]
 {hyperref}
\usepackage{breakurl}

\makeatletter
\newenvironment{svmultproof2}{\begin{proof}}{\smartqed\qed\end{proof}}
\newenvironment{svmultproof}{\begin{proof}}{\qed\end{proof}}

\usepackage{helvet}
\usepackage{courier}

\newcommand{\isdef}{\ensuremath{\stackrel{\text{def}}{=}}}

\newrefformat{prob}{Problem \ref{#1}}
\newrefformat{prop}{Proposition \ref{#1}}
\newrefformat{def}{Definition \ref{#1}}
\newrefformat{sub}{Section \ref{#1}}
\newrefformat{alg}{Algorithm \ref{#1}}
\newrefformat{con}{Conjecture \ref{#1}}
\newrefformat{cor}{Corollary \ref{#1}}
\newrefformat{rem}{Remark \ref{#1}}
\newrefformat{app}{Appendix \ref{#1}}
\newrefformat{ex}{Example \ref{#1}}
\newrefformat{enu}{Item \ref{#1}}

\DeclareMathOperator{\Op}{\mathfrak{D}}

\usepackage{arydshln}
\usepackage{multirow}

\makeatother

\begin{document}

\title*{Taylor Domination, Turán lemma, and Poincaré-Perron Sequences}

\author{Dmitry Batenkov and Yosef Yomdin}

\institute{Dmitry Batenkov \at Department of Computer Science, Technion - Israel Institute of Technology, Haifa 32000, Israel.
\email{batenkov@cs.technion.ac.il}\and Yosef Yomdin \at Department
of Mathematics, Weizmann Institute of Science, Rehovot 76100, Israel.
This author is supported by {\small ISF, Grants No. 639/09 and 779/13, and by the
Minerva foundation.} \email{yosef.yomdin@weizmann.ac.il}}

\maketitle

Abstract. {\small We consider ``Taylor domination'' property for an analytic function $f(z)=\sum_{k=0}^{\infty}a_{k}z^{k},$
in the complex disk $D_R$, which is an inequality of the form
\[
|a_{k}|R^{k}\leq C\ \max_{i=0,\dots,N}\ |a_{i}|R^{i}, \ k \geq N+1.
\]
This property is closely related to the classical notion of ``valency'' of $f$ in $D_R$. For $f$ - rational function
we show that Taylor domination is essentially equivalent to a well-known and widely used Turán's inequality on the sums
of powers.

Next we consider linear recurrence relations of the Poincaré type
\[
a_{k}=\sum_{j=1}^{d}[c_{j}+\psi_{j}(k)]a_{k-j},\ \ k=d,d+1,\dots,\quad\text{with }\lim_{k\rightarrow\infty}\psi_{j}(k)=0.
\]
We show that the generating functions of their solutions possess Taylor domination with explicitly
specified parameters. As the main example we consider moment generating functions, i.e. the Stieltjes transforms
\[
S_{g}\left(z\right)=\int\frac{g\left(x\right)\D x}{1-zx}.
\]
We show Taylor domination property for such $S_{g}$ when $g$ is
a piecewise D-finite function, satisfying on each continuity segment
a linear ODE with polynomial coefficients.}

\global\long\def\O{\Omega}
\global\long\def\C{\mathbb{C}}
\global\long\def\N{\mathbb{N}}
\global\long\def\R{\mathbb{R}}
\global\long\def\e{\varepsilon}
\global\long\def\Z{\mathbb{Z}}

\section{Introduction}

``Taylor domination'' for an analytic function $f(z)=\sum_{k=0}^{\infty}a_{k}z^{k}$ is an explicit bound of all its
Taylor coefficients $a_k$ through the first few of them. This property was classically studied, in particular, in relation
with the Bieberbach conjecture, finally proved in \cite{de1985proof}: for univalent $f$ always $|a_k| \leq k|a_1|$. See
\cite{bieberbach1955analytische,biernacki1936fonctions,Car,hayman1994multivalent} and references therein. To give an
accurate definition, let us assume the radius of convergence of the Taylor series for $f$ is $\hat R$, \
$0<\hat{R}\leqslant+\infty$.

\begin{definition}
\label{def:domination}Let a positive \emph{finite} $R\leq\hat{R},$
a natural $N$, and a positive sequence $S\left(k\right)$ of a subexponential growth be fixed. The function $f$ is said to
possess an $(N,R,S(k))$ - Taylor domination property if for each $k\geq N+1$ we have
\[
|a_{k}|R^{k}\leqslant S(k)\ \max_{i=0,\dots,N}|a_{i}|R^{i}.
\]
For $S\left(k\right)\equiv C$ a constant we shall call this property $\left(N,R,C\right)$-Taylor domination.
\end{definition}

The parameters $(N,R,S(k))$ of Taylor domination are not defined uniquely. In fact, each nonzero analytic function $f$
possesses this property, with $N$ being the index of its first nonzero Taylor coefficient $a_k$:

\begin{proposition}\label{prop:Root}
If $0<\hat{R}\leqslant+\infty$
is the radius of convergence of $f\left(z\right)=\sum_{k=0}^{\infty}a_{k}z^{k}$, with $f\not\not\equiv0$, then for each
finite and positive $0<R\leqslant\hat{R}$, $f$ satisfies the $\left(N,R,S\left(k\right)\right)$-Taylor domination
property with $N$ being the index of its first nonzero Taylor coefficient, and $S\left(k\right)= R^k|a_k|(|a_{N}|R^{N})^{-1},$
for $k > N$.

Conversely, let $f$ possess $\left(N,R,S\left(k\right)\right)$-domination. Then the series $\sum_{k=0}^{\infty}a_{k}z^{k}$
converges in a disk of radius $R^{*}$ satisfying $R^{*}\geqslant R$. For each $R'<R$ the function $f$ possesses
$(N,R',C)$-domination, with the constant $C$ depending on $R'/R$ and on the sequence $S(k)$.
\end{proposition}

\begin{svmultproof2}
It follows just by noticing that $\hat R^{-1}= \lim \sup_{k\rightarrow \infty} |a_k|^{1\over k}$. This implies that for each
$R\leq \hat R$ the sequence $S\left(k\right)= R^k|a_k|(|a_{N}|R^{N})^{-1}$ has a subexponential growth.
\end{svmultproof2}
Consequently, the Taylor domination becomes really interesting only for those {\it families} of analytic functions $f$ where
we can specify the parameters $N, \ R, \ S(k)$ in an explicit and uniform way.

\smallskip

In this paper we study Taylor domination for $a_k$ generated by linear non-stationary homogeneous recurrence relations of a
fixed length:

\begin{equation}
a_{k}=\sum_{j=1}^{d}c_{j}(k)\cdot a_{k-j},\ \ k=d,d+1,\dots,\label{eq:mainrec}
\end{equation}

If for $j=1,\ldots, d$ the coefficients $c_j(k)$ have a form $c_j(k)=c_{j}+\psi_{j}(k)$, with fixed $c_j$ and with
$\lim_{k\rightarrow\infty}\psi_{j}(k)=0$, then recurrence relation \eqref{eq:mainrec} is said to be
a \emph{linear recurrence relation of Poincaré type }(see \cite{perron1921summengleichungen,poincare1885equations}).

\smallskip

We start with recurrence relations \eqref{eq:mainrec} with constant coefficients. Those produce Taylor coefficients of
rational functions $f(z)$. Surprisingly, Taylor domination in this case turns out to be a nontrivial fact, essentially
equivalent to an important and widely used inequality on the sums of powers - Turan's lemma
(\cite{turan1953neue,turan1984new,nazarov1994local,Ineq}).

\smallskip

In Section \ref{sub:rational-gf-turan} we prove this equivalence, obtaining, as a byproduct, a very short proof of Turan's
lemma from the classical Biernacki theorem (\cite{biernacki1936fonctions}) on multivalent functions, and from Bezout bound
on zeroes of rational functions.

\smallskip

In Section \ref{sec:main-res} we obtain an explicit Taylor domination for solutions of general recurrence relations
\eqref{eq:mainrec} with uniformly bounded coefficients. No additional information, besides the size of $c_j(k),$ is
required. So we get this result, and hence (via Section \ref{sec:zeroes} below) also explicit bounds on the number of
zeroes, for a fairly wide class of analytic functions. On the other hand, in contrast with Turan's lemma, the disk
 where we get Taylor domination is usually significantly smaller than the expected disk of convergence.

\smallskip

In Section \ref{sec:poincare-type-rec} we obtain some more accurate results for recurrence relations of Poincaré type.
This includes (only partially explicit) Taylor domination in the full disk of convergence, and an explicit and uniform
(depending only the size of perturbations $\psi_{j}(k)$) one in a smaller disk.

\smallskip

Finally, in Section \ref{sec:dfinite}, we consider moments of ``$D$-finite'' functions, produce a linear recurrence relation for them,
describe the cases when it is of Poincaré type, and provide some conclusions concerning Taylor domination.

\smallskip

The authors would like to thank O. Friedland for a careful reading of the first version of this paper, and for numerous
important remarks and suggestions.

\section{Taylor domination and counting zeroes}\label{sec:zeroes}

Taylor domination allows us to compare the behavior of $f(z)$ with
the behavior of the polynomial $P_{N}(z)=\sum_{k=0}^{N}a_{k}z^{k}$.
In particular, the number of zeroes of $f$ can be easily bounded
in this way. In one direction the bound is provided by the classical results of \cite{biernacki1936fonctions,Car}.
To formulate them, we need the following definition (see \cite{hayman1994multivalent} and references therein):

\begin{definition}
\label{def:pvalent}A function $f$ regular in a domain $\O\subset{\mathbb{C}}$
is called $p$-valent there, if for any $c\in{\mathbb{C}}$ the number
of solutions in $\O$ of the equation $f(z)=c$ does not exceed $p$.
\end{definition}

\begin{theorem}[Biernacki, 1936, \cite{biernacki1936fonctions}]
\label{thm:bier}If $f$ is $p$-valent in the disk $D_{R}$ of radius $R$ centered at $0\in{\mathbb{C}}$ then for each
$k\geq p+1$
\[
|a_{k}|R^{k}\le(A(p)k/p)^{2p}\max_{i=1,\ldots,p}|a_{i}|R^{i},
\]
where $A(p)$ is a constant depending only on $p$.

\end{theorem}
In our notations, \prettyref{thm:bier} claims that a function $f$
which is $p$-valent in $D_{R},$ possesses a $(p,R,(Ak/p)^{2p})$
- Taylor domination property.

For univalent functions, i.e. for $p=1,\ R=1,$ \prettyref{thm:bier} gives $|a_{k}|\le A(1)^2k^2|a_{1}|$
for each $k$, while the sharp bound of the Bieberbach conjecture is $|a_{k}|\le k|a_{1}|$.

A closely related result (obtained somewhat earlier) is the following:

\begin{theorem}[Cartwright, 1930, \cite{Car}]\label{thm:Car}If $f$ is $p$-valent in the disk $D_{1}$ of radius
$1$ centered at $0\in{\mathbb{C}}$ then for $|z| < 1$
\[
|f(z)| < B(p)\max_{i=0,\ldots,p}|a_{i}| (1-|z|)^{-2p},
\]
where $B(p)$ is a constant depending only on $p$.
\end{theorem}

Various forms of inverse results to \prettyref{thm:bier}, \prettyref{thm:Car} are known. In particular, an explicit bound 
for the number of zeroes of $f$ possessing Taylor domination can be obtained by combining \prettyref{prop:Root} and 
Lemma 2.2 from \cite{roytwarf1997bernstein}:

\begin{theorem}\label{thm:Roy.Yom} Let the function $f$ possess an $(N,R,S(k))$ - Taylor domination property. Then for each
$R'<R$, $f$ has at most $M(N,\frac{R'}{R},S(k))$ zeros in $D_{R'}$, where $M(N,\frac{R'}{R},S(k))$ is a function depending only on $N$, $\frac{R'}{R}$ and on the
sequence $S(k)$, satisfying $\lim_{{{R'}\over R}\to 1}M=\infty$ and $M(N,\frac{R'}{R},S)=N$ for $\frac{R'}{R}$ sufficiently small.
\end{theorem}
We can replace the bound on the number of zeroes of $f$ by the bound on its valency, if we exclude $a_0$ in the definition of
Taylor domination (or, alternatively, if we consider the derivative $f'$ instead of $f$).

\section{\label{sub:rational-gf-turan}Taylor domination for rational functions and Turán's lemma}

We shortly recall some basic facts concerning Taylor coefficients of rational functions. Consider a rational function
$R(z)= {{P(z)}\over {Q(z)}}$ with $Q(z)= 1-\sum_{j=1}^d c_jz^j$ and $\deg P(z) \leq d-1$. To simplify the presentation we
shall assume that all the roots $s_1,\ldots,s_d$ of $Q$ are pairwise different (and they are clearly nonzero since $Q(0)=1$). All the results below remain valid
in the general case of multiple roots. Now $R(z)$ can be represented as a sum of elementary fractions:

\begin{equation}
R(z)=\sum_{j=1}^{d} {{\alpha_j}\over {s_j-z}}= \sum_{j=1}^{d} {{\beta_j}\over {1-\sigma_jz}}, \ \
with \ \ \beta_j={{\alpha_j}\over {s_j}}, \ \sigma_j={1\over {s_j}}. \label{eq:elem.fr}
\end{equation}
Developing into geometric progressions, we obtain

\begin{equation}\label{eq:Prony}
R(z)=\sum_{k=0}^{\infty} a_k z^k, \ \ with \ \ a_k=\sum_{j=1}^{d} \beta_j\sigma_j^k.
\end{equation}
Assuming that all $\alpha_j$ are nonzero, the radius of convergence of this series is
$R\isdef\min_{i=1,\dots d}\left|\sigma_{i}^{-1}\right|,$ which is the distance from the origin to the nearest pole of $R(z)$.

\smallskip

It is well known that the Taylor coefficients $a_k$ of $R(z)$ satisfy a linear recurrence relation with constant coefficients
\begin{equation}
a_{k}=\sum_{j=1}^{d}c_{j}a_{k-j},\ \ k=d,d+1,\dots,\label{eq:rec.Const.Coef}
\end{equation}
where $c_j$ are the coefficients of the denominator $Q(z)$ of $R(z)$. Conversely, for any initial terms $a_0,\ldots,a_{d-1}$ the
solution sequence of \eqref{eq:rec.Const.Coef} forms a sequence of the Taylor coefficients $a_k$ of a rational function $R(z)$
as above. The equation $\sigma^d-\sum_{j=1}^d c_j\sigma^{d-j}$ is called the characteristic equation, and its roots
$\sigma_1,\ldots,\sigma_d$ are called the characteristic roots of \eqref{eq:rec.Const.Coef}.

\smallskip

Taylor domination property for rational functions is provided by the following theorem, which is, essentially, equivalent to the
``first Turán lemma'' (\cite{turan1953neue,turan1984new,Ineq}, see also \cite{nazarov1994local}):

\begin{theorem}\label{thm:turan} Let $\left\{ a_{j}\right\} _{j=1}^{\infty}$ satisfy
recurrence relation \eqref{eq:rec.Const.Coef} and let $\sigma_1,\ldots,\sigma_d$ be its characteristic roots. Put
$R\isdef\min_{i=1,\dots d}\left|\sigma_{i}^{-1}\right|.$ Then for each $k\geq d$

\begin{equation}
\left|a_{k}\right|R^{k}\leq \ Q(k,d)\max_{i=0,\dots,d-1}\ |a_{i}|R^{i},\label{eq:tur}
\end{equation}
where $Q(k,d)=[2e({k\over d}+1)]^d$.
\end{theorem}
\begin{svmultproof2} The original Turán's result is as follows: let $b_j,z_j, \ j=1,\ldots,d,$ be given complex numbers, with
$\min_j |z_j|=1$. Denoting $B\isdef \left(b_1,\dots,b_d\right)$ and $Z\isdef \left(z_1,\dots,z_d\right)$, define $g_\nu$ as the power sum $g_\nu(B,Z)=\sum_{j=1}^d b_jz_j^\nu.$

\begin{theorem}[Turán, 1953, \cite{turan1953neue}]\label{thm:turan1} For each natural $m$ we have

\begin{equation}
|b_1+\ldots+b_d| \leq Q(m,d)\max_{\nu=m+1,\dots,m+d}\ |g_\nu|.\label{eq:tur1}
\end{equation}
where $Q(m,d)=[2e({m\over d}+1)]^d$.
\end{theorem}
Put $k=m+d$. We immediately obtain that for any $F=\left(f_1,\dots,f_d\right)$ and $W=\left(w_1,\dots,w_d\right)$ with $\max |w_j|=1$ we have
\begin{equation}
|g_k(F,W)| \leq Q(m,d)\max_{i=0,\dots,d-1}\ |g_i(F,W)|,\label{eq:tur2}
\end{equation}
by applying Theorem \ref{thm:turan} with $z_j=w_j^{-1}$ and $b_j=f_jw_j^k$.

Now we return to the sequence $a_k$ satisfying recurrence relation \eqref{eq:rec.Const.Coef}. Put $S=\left(\sigma_1R,\dots,\sigma_dR\right)$ and $D=\left(\beta_1,\dots,\beta_d\right)$.
Then, according to \eqref{eq:Prony}, we get $|g_k(D,S)|=|a_k|R^k$. Clearly, $\max |\sigma_i R|=1$ and so the inequality \eqref{eq:tur2}  can be applied, giving the required inequality \eqref{eq:tur}.
\end{svmultproof2}
\prettyref{thm:turan} provides an interpretation of Turán's lemma as a statement about Taylor domination for rational functions. 
This fact allows one to give a very short proof of Turán's lemma, albeit with a less sharp bound. Indeed, by Bezout theorem, rational 
functions of degree $d$ are globally $d$-valent. In particular, $R(z)$ as above is $d$-valent in its maximal disk of convergence. 
Applying Biernacki's \prettyref{thm:bier}, we get an inequality

\begin{equation}
\left|a_{k}\right|R^{k}\leq \ [{{A(d)k}\over d}]^{2d} \max_{i=1,\dots,d}\ |a_{i}|R^{i}.\label{eq:bier1}
\end{equation}
The transformations as above (shifted by $1$) show that this is equivalent to Turán's lemma, with $Q(k,d)=[{{A(d)k}\over d}]^{2d}$.
This is less sharp than Turan's expression. Notice, that the best possible constant is given in \cite{Brui}:

$$
Q(k,d)= \sum_{j=0}^{d-1}(_j^{k+j})2^j.
$$
\prettyref{thm:turan} provides a uniform Taylor domination for rational functions in their maximal disk of convergence $D_{R}$,
in the strongest possible sense. Indeed, after rescaling to the unit disk $D_{1}$ the parameters of \eqref{eq:tur} depend only on
the degree of the function, but not on its specific coefficients.

\smallskip

Turán's lemma can be considered as a result on exponential polynomials, and in this form it was a starting point for many deep
investigations in Harmonic Analysis, Uncertainty Principle, Analytic continuation, Number Theory
(see \cite{Ineq,nazarov1994local,turan1953neue,turan1984new} and references therein). Recently some applications in Algebraic Sampling 
were obtained, in particular, estimates of  robustness of non-uniform sampling of ``spike-train'' signals 
(\cite{Sampl,friedland2011observation}). One can hope that apparently new connections of Turán's lemma with Taylor domination, 
presented in this paper, can be further developed.

\smallskip

A natural open problem, motivated by \prettyref{thm:turan}, is a possibility to extend uniform Taylor domination in the maximal disk
of convergence $D_{R}$, as provided by \prettyref{thm:turan} for rational functions, to wider classes of generating functions of
Poincaré type recurrence relations. Some initial examples in this direction were provided in \cite{yomdin2014Bautin}, via techniques
of Bautin ideals. An inequality closely related to Turán's lemma was obtained in \cite{yomdin2010sing.Prony}, via techniques of finite
differences. Presumably, the latter approach can be extended to Stieltjes transforms of certain natural classes of functions, providing
uniform Taylor domination in the maximal disk of convergence.

\section{\label{sec:main-res}Recurrence relations with uniformly bounded coefficients}

We will denote by ${\cal R}$ a non-stationary linear recurrence relation of the form \eqref{eq:mainrec}. Let ${\cal S}$ be
the class of such ${\cal R}$ with uniformly bounded coefficients $c_j(k)$. We would like to write the bounds on $c_j(k)$ in a form
\[
\left|c_{j}\left(k\right)\right|\leqslant K\rho^{j},\qquad j=1,\dots,d, \ k=d, d+1, \ldots \ ,
\]
for certain positive constants $K,\rho$. So for each ${\cal R}\in {\cal S}$ we define $K({\cal R})$ and $\rho({\cal R})$ to be the pair
of constants providing the required bounds on $c_j(k)$, for which the product $\nu({\cal R})=(2K({\cal R})+2)\cdot \rho({\cal R})$ is
minimal possible. We put $R({\cal R})\isdef \nu({\cal R})^{-1}.$

\begin{theorem}\label{thm:main-result} Let $\left\{ a_{k}\right\} _{k=0}^{\infty}$ be a solution of the recurrence relation
${\cal R}\in {{\cal S}}$. Put $K=K({\cal R}), \ \rho=\rho({\cal R}), \ R=R({\cal R}).$ Then the series $f(z)=\sum_{k=0}^\infty a_kz^k$
converges in the open disk $D_R$ and possesses there $(d-1,R,(2K+2)^{d-1})$ Taylor domination. 
\end{theorem}
\begin{svmultproof2} Put $\hat R\isdef\rho^{-1}$. First we  show that for each $k\geqslant d$ we have

\begin{equation}
\left|a_{k}\right|\hat R^{k}\leqslant\left(2K+2\right)^{k}\max_{i=0,\dots,d-1}\left|a_{i}\right|\hat R^{i}.\label{eq:main-res-dom}
\end{equation}
The proof is by induction on $k$. Denote 
\begin{eqnarray*}
M & \isdef & \max_{i=0,\dots,d-1}\left|a_{i}\right|\hat R^{i},\\
\eta & \isdef & 2K+2,
\end{eqnarray*}
and assume that $\left|a_{\ell}\right|\hat R^{\ell}\leqslant\eta^{\ell}M,\quad\ell\leqslant k-1$. We have
\[
|a_{k}|\hat R^{k}=\hat R^{k}|\sum_{j=1}^{d}c_{j}(k)a_{k-j}|\leq K \hat R^{k}\sum_{j=1}^{d}|a_{k-j}|\rho^{j}=K\sum_{j=1}^{d}|a_{k-j}|\hat R^{k-j}.
\]
By the inductive assumption $\left|a_{k-j}\right|\hat R^{k-j}\leqslant\eta^{k-j}M$, therefore we conclude that
$$
\left|a_{k}\right|\hat R^{k} \leqslant KM\sum_{j=1}^{d}\eta^{k-j}=KM\eta^{k-1}\sum_{s=0}^{d-1}\eta^{-s} \leqslant 2KM\eta^{k-1} < \eta^{k}M.
$$
This completes the induction step and proves inequality \eqref{eq:main-res-dom}. Since by definition $R={{\hat R}\over \eta}$, dividing
\eqref{eq:main-res-dom} by $\eta^{k}$ gives 
\begin{eqnarray*}
\left|a_{k}\right|R^{k} & \leqslant & \max_{i=0,\dots,d-1}\left|a_{i}\right|R^{i}\eta^{i}\leqslant\eta^{d-1}
\max_{i=0,\dots,d-1}\left|a_{i}\right|R^{i}.
\end{eqnarray*}
This is precisely the definition of $\left(d-1,R,\eta^{d-1}\right)$-Taylor domination, which, in turn, implies convergence in the open
disk $D_R$.
\end{svmultproof2}
By a proper rescaling, \prettyref{thm:main-result} can be easily extended to non-stationary linear recurrences with a subexponential
(or exponential) growth of the coefficients $c_j(k)$. Consequently, generating functions of such recurrences allow for explicit bounds
on their valency. On the other hand, a drawback of this approach is that in the case of linear recurrences with constant coefficients
(and for Poincar\'e-type recurrences - see below) the disk $D_R$ where the uniform Taylor domination is guaranteed, is much smaller
than the true disk of convergence.

\smallskip

\noindent{\bf Example 1.} Consider a recurrence ${\cal R}$ with constant coefficients of form \eqref{eq:rec.Const.Coef}:
$a_{k}=\sum_{j=1}^{d}c_{j}a_{k-j},\ \ k=d,d+1,\dots$. Let $\sigma_1,\ldots,\sigma_d$ be its characteristic roots, i.e. the roots of its
characteristic equation $\sigma^d-\sum_{j=1}^d c_j\sigma^{d-j}=0$. Put $\rho=\max_j |\sigma_j|$. Then for each $j=1,\dots,d$ we have
\[
\left|c_{j}\right|=\left|e_{j}\left(\sigma_{1},\dots,\sigma_{d}\right)\right|\leqslant{d \choose j}\rho{}^{j}\leqslant2^{d}\rho{}^{j},
\]
where $e_{j}\left(\cdot\right)$ is the elementary symmetric polynomial of degree $j$ in $d$ variables. We conclude that
$K({\cal R})\leq 2^d, \ \rho({\cal R}) \leq \rho, \ R({\cal R})\geq 2^{-d}\rho^{-1}.$ It is easy to see that these bounds are sharp.
So the radius of convergence is $\rho^{-1},$ while \prettyref{thm:main-result} provides Taylor domination only in the concentric disk
of a $(2^d+2)$ times smaller radius. In the next section we discuss in some detail the problem of extending uniform Taylor domination to the 
full disk of convergence for Poicar\'e-type recurrences.

\section{\label{sec:poincare-type-rec}Taylor domination for Poincaré-type recurrences}

Now we consider recurrence relations ${\cal R}$ of Poincaré-type, i.e. of the form

\begin{equation}
a_{k}=\sum_{j=1}^{d}[c_{j}+\psi_{j}(k)]\cdot a_{k-j},\ \ k=d,d+1,\dots, \ \lim_{k\rightarrow\infty}\psi_{j}(k)=0. \label{eq:mainrec1}
\end{equation}
We denote this class by ${\cal S}_P.$ The characteristic polynomial $\sigma^d-\sum_{j=1}^d c_j\sigma^{d-j}=0$ and the characteristic
roots $\sigma_1,\ldots,\sigma_d$ of ${\cal R}\in {\cal S}_P$ are those of its constant part. We put $\rho({\cal R})=\max_j |\sigma_j|.$

\smallskip

The asymptotic behaviour of the solutions of such recurrences has been extensively studied,
starting from Poincaré's own paper \cite{poincare1885equations}. A comprehensive overview of the subject can be found in e.g.
\cite[Chapter 8]{elaydi2005ide}. The general idea permeating these studies is to compare the solutions of \eqref{eq:mainrec} to the
solutions of the corresponding unperturbed recurrence relation \eqref{eq:rec.Const.Coef}. In this last case, as we have seen above,
for some $\sigma_{j}$

\[
\limsup_{k\to\infty}\sqrt[k]{\left|a_{k}\right|}=\left|\sigma_{j}\right|.
\]
O.Perron proved in \cite{perron1921summengleichungen} that this relation
holds for a general recurrence of Poincaré type, but with an additional
condition that $c_{d}+\psi_{d}\left(k\right)\neq0$ for all $k\in\N$.
In \cite{pituk2002more} M.Pituk removed this restriction, and proved
the following result.

\begin{theorem}[Pituk's extension of Perron's Second Theorem, \cite{pituk2002more}]
\label{thm:pituk-scalar}Let $\left\{ a_{k}\right\} _{k=0}^{\infty}$
be any solution to a recurrence relation ${\cal R}$ of Poincaré class
${\cal S}_P$. Then either $a_{k}=0$ for $k\gg1$ or
\[
\limsup_{k\to\infty}\sqrt[k]{\left|a_{k}\right|}=\left|\sigma_{j}\right|,
\]
where $\sigma_{j}$ is one of the characteristic roots of ${\cal R}$.
\end{theorem}
This result, together with \prettyref{prop:Root}, implies the following:
\begin{theorem}\label{thm:poinc-dom-maximal-disk} Let $\left\{ a_{k}\right\} _{k=0}^{\infty}$
be any nonzero solution to a recurrence relation ${\cal R}$ of Poincaré
class ${\cal S}_P$ with initial data $\bar{a}$, and let $R$ be the
radius of convergence of the generating function $f\left(z\right)$.
Then necessarily $R>0$, and in fact $R=\left|\sigma\right|^{-1}$
where $\sigma$ is some (depending on $\bar{a}$) characteristic root
of ${\cal R}$. Consequently, $f$ satisfies $\left(d-1,R,S\left(k\right)\right)$-Taylor
domination with $S\left(k\right)$ as defined in \prettyref{prop:Root}.\end{theorem}
\begin{svmultproof2}
The only thing left to show is that $a_{m}\neq0$ for some $m=0,1,\dots,d-1$.
Assume on the contrary that
\[
a_{0}=a_{1}=\dots=a_{d-1}=0.
\]
Plugging this initial data into the recurrence \eqref{eq:mainrec1},
we immediately conclude that $a_{k}=0$ for all $k\in\N$, a contradiction.
\end{svmultproof2}
Taylor domination in the maximal disk of convergence provided by \prettyref{thm:poinc-dom-maximal-disk}, is not effective.
Indeed, \prettyref{prop:Root} guarantees that the sequence $S\left(k\right)=R^k|a_k|\cdot (\max_{i=0}^{d-1}R^i|a_i|)^{-1}$ is of
subexponential growth but gives no further information on it. We can pose a natural question in this direction. For a sequence
$\Delta=\{\delta_k\}$ of positive numbers tending to zero, consider a subclass ${\cal S}_{P,\Delta}$ of ${\cal S}_P$,
consisting of ${\cal R}\in {\cal S}_P$ with $|\psi_{j}(k)|\leq \delta_k\cdot \rho({\cal R})^j, \ j=1,\ldots,d, \ k=d,d+1,\ldots$

\smallskip

\noindent{\bf Problem 1.} {\it Do solutions of recurrence relations ${\cal R}\in {\cal S}_{P,\Delta}$ possess \ \ \
$(N,R,S(k))$-Taylor domination in the maximal disk of convergence $D_R$, with $S(k)$ depending only on $d$ and $\Delta$?
Is this true for specific $\Delta$, in particular, for $\Delta= \{1,{1\over 2},{1\over 3},...,\},$ as it occurs in most of
examples (solutions of linear ODE's, etc.)?}

\medskip

Taking into account well known difficulties in the analysis of Poincaré-type recurrences, this question may be tricky.
Presumably, it can be easier for $\Delta$ with $\sum_{k=1}^\infty \delta_k < \infty$.

\smallskip

In the remainder of this section we prove a version of \prettyref{thm:main-result} for Poincaré-type recurrences ${\cal R}$. It
provides Taylor domination in a smaller disk, but with explicit parameters, expressed in a transparent way through the constant
part of ${\cal R}$, and through the size of the perturbations.

\begin{theorem}\label{thm:Poinc.Rec} Let $\left\{a_{k}\right\} _{k=0}^{\infty}$ satisfy a fixed recurrence
${\cal R}\in {\cal S}_P$. Put $\rho\isdef \rho({\cal R}), \ R=2^{-(d+3)}\rho^{-1}.$ Let $\hat N$ be the minimal of
the numbers $n$ such that for all $k>n$ we have $|\psi_{j}(k)|\leq 2^d\rho^j, \ j=1,\dots,d$. We put $N=\hat{N}+d,$ and
$C=2^{(d+3)N}.$

Then $\left\{a_{k}\right\} _{k=0}^{\infty}$ possesses $(N,R,C)$-Taylor domination property.
\end{theorem}
\begin{svmultproof2}
By the calculations in Example 1 above, we have for the constant part of ${\cal R}$ the bounds $|c_j|\leq 2^{d}\rho{}^{j}.$
Since for $k>N$ we have $|\psi_{j}(k)|\leq 2^d\rho^j,$ we obtain for such $k$ that $|c_j(k)|\leq 2^{d+1}\rho{}^{j}.$
Now repeating the proof of \prettyref{thm:main-result} (with the starting point of the recurrence shifted by $\hat N$) we
complete the proof.
\end{svmultproof2}
\prettyref{thm:Poinc.Rec} provides a partial answer to Problem 1:

\begin{corollary}\label{cor:poinc-taylor-dom}
Let $\Delta=\{\delta_k\}$ be a sequence of positive numbers tending to zero. Define $\hat N(\Delta)$ as a minimal number $n$
such that for $k>n$ we have $\delta_k\leq 2^d$. Then for each ${\cal R}\in {\cal S}_{P,\Delta}$ solution sequences of ${\cal R}$
possess $(N,R,C)$-Taylor domination, where $N=\hat N(\Delta)+d, \ R=2^{-(d+3)}\rho({\cal R})^{-1}, \ C=2^{(d+3)N}.$
\end{corollary}
\begin{svmultproof2}
By definition of the class ${\cal S}_{P,\Delta}$ the number $\hat N(\Delta)$ coincides with $\hat N$ in \prettyref{thm:Poinc.Rec}.
Application of this theorem completes the proof.
\end{svmultproof2}

One can consider at least two possible approaches to the extension of these results to the full
disk of convergence $D_{R}$. First, asymptotic expressions in \cite{bodine2004asymptotic,pituk2002more}
may be accurate enough to provide an inequality of the desired form. If this is a case, it remains to get explicit bounds in
these asymptotic expressions.

Second, one can use a ``dynamical approach'' to recurrence relation
\eqref{eq:mainrec} (see \cite{borcea2011parametric,coppel1971dichotomies,kloeden2011non,potzsche2010geometric,yom.nonaut.dyn} 
and references therein). We consider \eqref{eq:mainrec} as a non-autonomous
linear dynamical system $T$. A ``non-autonomous diagonalization'' of $T$ is a sequence ${\cal L}$ of linear changes of variables,
bringing this system to its ``constant model'' $T_{0}$, provided by the limit recurrence relation \eqref{eq:rec.Const.Coef}.

If we could obtain a non-autonomous diagonalization ${\cal L}$ of $T$ with an explicit bound on the size of the linear changes of
variables in it, we could get the desired inequality as a pull-back, via ${\cal L}$, of the Turán inequality for $T_{0}$. There
are indications that the second approach may work in the classes ${\cal S}_{P,\Delta}$, for $\Delta$ with a finite sum.

\section{\label{sec:dfinite}Piecewise D-finite functions}

\global\long\def\np{p}
\global\long\def\ff#1#2{\left(#1\right)_{#2}}
\global\long\def\mvi{\ensuremath{\Lambda}}

In this section we investigate a certain class of power series, defined by the Stieltjes integral transforms
\begin{equation}
f\left(z\right)=S_{g}\left(z\right)=\int_a^b\frac{g\left(x\right)\D x}{1-zx},\label{eq:def-stieltjes}
\end{equation}
where $g\left(x\right)$ belongs to the class of the so-called \emph{piecewise
D-finite functions} \cite{bat2008}, which are solutions of linear ODEs with polynomial coefficients, possessing a finite number of
discontinuities of the first kind.

Using the expansion $\left(1-zx\right)^{-1}=\sum_{k=0}^{\infty}\left(zx\right)^{k}$ for $\left|z\right|<\frac{1}{\left|x\right|}$,
 we obtain the following useful representation of $S_{g}\left(z\right)$:

\begin{proposition}
Let $g:\left[a,b\right]\to\R$ be bounded and integrable on $\left[a,b\right]$.
Then the Stieltjes transform \eqref{eq:def-stieltjes} of $g$ is
regular at the origin, and it is given by the moment-generaing function
\[
S_{g}\left(z\right)=\sum_{k=0}^{\infty}m_{k}z^{k},\quad\text{where }m_{k}\isdef\int_{a}^{b}x^{k}g\left(x\right)\D x.
\]
\end{proposition}
\begin{definition}
A real-valued bounded integrable function $g:\left[a,b\right]\to\R$
is said to belong to the class ${\cal PD}\left(\Op,\np\right)$ if
it has $0\leqslant\np<\infty$ discontinuities (not including the
endpoints $a,b$) of the first kind, and between the discontinuities
it satisfies a linear homogeneous ODE with polynomial coefficients
$\Op g=0$, where
\[
\Op=\sum_{j=0}^{n}p_{j}\left(x\right)\left(\frac{\D}{\D x}\right)^{j},\quad p_{j}\left(x\right)=\sum_{i=0}^{d_{j}}a_{i,j}x^{i}.
\]

\end{definition}
Obtaining uniform Taylor domination for $S_{g},$ where $g$ belongs to particular subclasses of ${\cal PD}$ (in particular, $g$
being piecewise algebraic), is an important problem with direct applications in Qualitative Theory of ODEs
(see \cite{briskin2010center} and references therein). In this paper we provide initial results in this direction.

\smallskip

Let $g\in{\cal PD}\left(\Op,\np\right)$, with $\Op$ as above. Denote the discontinuities of $g$ by
$a=x_{0}<x_{1}<\dots<x_{\np}<x_{\np+1}=b$. In what follows, we shall use some additional notation. Denote for each $j=0,\dots,n,$
\ \ $\alpha_{j}\isdef d_{j}-j.$ Let $\alpha\isdef\max_{j}\alpha_{j}$, and denote for each $\ell=-n,\dots,\alpha$
\begin{equation}
q_{\ell}\left(k\right)\isdef\sum_{j=0}^{n}\left(-1\right)^{j}a_{\ell+j,j}\ff{k+\ell+j}j.\label{eq:ql-def}
\end{equation}
where $\ff ab=a\left(a-1\right)\times\dots\times\left(a-b+1\right)=\frac{\Gamma\left(a\right)}{\Gamma\left(b\right)}$
is the Pochhammer symbol for the falling factorial.

Our approach is based on the following result:
\begin{theorem}[\cite{bat2008}] Let $g\in{\cal PD}\left(\Op,\np\right)$. Then the moments $m_{k}=\int_{a}^{b}g\left(x\right)\D x$
satisfy the recurrence relation

\begin{equation}
\sum_{\ell=-n}^{\alpha}q_{\ell}\left(k\right)m_{k+\ell}=\e_{k},\quad k=0,1,\dots,\label{eq:moments-main-rec}
\end{equation}
where
\[
\e_{k}=\sum_{\ell=0}^{\np+1}\sum_{j=0}^{n-1}x_{\ell}^{k-j}\ff kjc_{\ell,j},
\]
with each $c_{\ell,j}$ being a certain homogeneous bilinear form in the two sets of variables
\begin{align*}
\; & \{p_{m}(x_{\ell}),p'_{m}(x_{\ell}),\dots,p_{m}^{(n-1)}(x_{\ell})\}_{m=0}^{n},\\
\; & \{g(x_{\ell}^{+})-g(x_{\ell}^{-}),\; g'(x_{\ell}^{+})-g'(x_{\ell}^{-}),\dots,g^{(n-1)}(x_{\ell}^{+})-g^{(n-1)}(x_{\ell}^{-})\}.
\end{align*}
\end{theorem}

The recurrence \eqref{eq:moments-main-rec} is inhomogeneous, and the coefficient of the highest moment is different from one. Accordingly, 
we first transform \eqref{eq:moments-main-rec} into a homogeneous matrix recurrence. Next, we give conditions for this last recurrence to 
be of Poincar\'e type. Finally, we apply an appropriate version of the results in Section \ref{sec:poincare-type-rec} to get Taylor
domination.

\smallskip

It is well-known that the sequence $\left\{ \e_{k}\right\} _{k=0}^{\infty}$
satisfies a recurrence relation ${\cal S}$ of the form \eqref{eq:rec.Const.Coef}
with constant coefficients, whose characteristic roots are precisely
$\left\{ x_{0},\dots,x_{p+1}\right\} $, each with multiplicity $n$.
Let the characteristic polynomial $\Theta_{{\cal S}}\left(z\right)$
of degree $\tau\isdef n\left(\np+2\right)$ be of the form
\begin{equation}
\Theta_{{\cal S}}\left(\sigma\right)=\prod_{j=0}^{\np+1}\left(\sigma-x_{j}\right)^{n}=\sigma^{\tau}-\sum_{i=1}^{\tau}b_{i}\sigma^{\tau-i},\label{eq:char-poly-moments-const-part}
\end{equation}
then
\[
\e_{k}=\sum_{j=1}^{\tau}b_{j}\e_{k-j},\quad k=\tau,\tau+1,\dots.
\]
Rewrite this last recurrence as
\begin{equation}
\e_{k+\tau}=\sum_{j=0}^{\tau-1}b_{\tau-j}\e_{k+j},\quad k=0,1,\dots.\label{eq:rec-for-rhs-nonhom}
\end{equation}
Now denote the vector function $\vec{w}\left(k\right):\N\to\C^{\alpha+n+\tau}$
as
\[
\vec{w}\left(k\right)\isdef\begin{bmatrix}m_{k-n}\\
\vdots\\
m_{k+\alpha-1}\\
\e_{k}\\
\vdots\\
\e_{k+\tau-1}
\end{bmatrix}.
\]
Then by \eqref{eq:moments-main-rec} and \eqref{eq:rec-for-rhs-nonhom}
we see that $\vec{w}\left(k\right)$ satisfies the linear system
\begin{equation}
\vec{w}\left(k+1\right)=\begin{bmatrix}0 & 1 & 0 & \dots & 0 & \multicolumn{4}{c}{\multirow{3}{*}{\ensuremath{\vec{0}^{\left(\alpha+n-1\right)\times\tau}}}}\\
0 & 0 & 1 & \dots & 0\\
\dots\\
-\frac{q_{-n}\left(k\right)}{q_{\alpha}\left(k\right)} & -\frac{q_{-n+1}\left(k\right)}{q_{\alpha}\left(k\right)} & \dots &  & -\frac{q_{\alpha-1}\left(k\right)}{q_{\alpha}\left(k\right)} & \frac{1}{q_{\alpha}\left(k\right)} & 0 & \dots &  & 0\\
 &  &  &  &  & 0 & 1 & 0 & \dots & 0\\
 &  &  &  &  & 0 & 0 & 1 & \dots & 0\\
\multicolumn{5}{c}{\ensuremath{\vec{0}^{\tau\times\left(\alpha+n\right)}}} & \dots\\
 &  &  &  &  & b_{\tau} & b_{\tau-1} & \dots &  & b_{1}
\end{bmatrix}\vec{w}\left(k\right).\label{eq:moment-rec-system}
\end{equation}

Now we can show Taylor domination for the Stieltjes transform $S_{g}\left(z\right)$,
utilizing the system version of \prettyref{thm:pituk-scalar}.
\begin{definition}
The vector function $\vec{y}\left(k\right):\N\to\C^{n}$ is said to
satisfy a linear system of Poincaré type, if
\begin{equation}
\vec{y}\left(k+1\right)=\left(A+B\left(k\right)\right)\vec{y}\left(k\right),\label{eq:system-poincare}
\end{equation}
where $A$ is a constant $n\times n$ matrix and $B\left(k\right):\N\to\C^{n\times n}$
is a matrix function satisfying $\lim_{k\to\infty}\|B\left(k\right)\|=0$.\end{definition}
\begin{theorem}[\cite{pituk2002more}]
\label{thm:pituk-system}Let the vector $\vec{y}\left(k\right)$
satisfy the perturbed linear system of Poincaré type \eqref{eq:system-poincare}.
Then either $\vec{y}\left(k\right)=\vec{0}\in\C^{n}$ for $k\gg1$
or
\[
\lim_{k\to\infty}\sqrt[k]{\|\vec{y}\left(k\right)\|}
\]
exists and is equal to the modulus of one of the eigenvalues of the
matrix $A$.\end{theorem}
\begin{lemma}
\label{lem:cond-for-poincare-syst}The recurrence system \eqref{eq:moment-rec-system}
is of Poincaré type if and only if
\begin{equation}
\alpha_{n}\geqslant\alpha_{j}\qquad j=0,1,\dots,n.\label{eq:poinc-sys-cond}
\end{equation}
\end{lemma}
\begin{svmultproof2}
Clearly, a necessary and sufficient condition for \eqref{eq:moment-rec-system}
being of Poincaré type is that
\[
\deg q_{\ell}\left(k\right)\leqslant\deg q_{\alpha}\left(k\right),\quad\ell=-n\dots,\alpha-1.
\]
We will show that this condition is equivalent to \eqref{eq:poinc-sys-cond}.

Recall the definition \eqref{eq:ql-def}. The highest power of $k$
in any $q_{\ell}\left(k\right)$ is determined by the maximal index
$j=0,\dots,n$ for which $a_{i,j}\neq0$ and $i-j=\ell$. Consider
$\ell=\alpha_{n}=d_{n}-n$, then obviously since $a_{d_{n},n}\neq0$
we have $\deg q_{\alpha_{n}}\left(k\right)=n$.
\begin{enumerate}
\item Now let's assume that for some $s<n$ we have $\alpha_{s}>\alpha_{n}$,
i.e. $d_{s}-s>d_{n}-n$, and consider the polynomial $q_{\alpha_{s}}\left(k\right)$:
\[
q_{\alpha_{s}}\left(k\right)=\sum_{j=0}^{n}\left(-1\right)^{j}a_{j+\alpha_{s},j}\ff{k+j+\alpha_{s}}j.
\]
By assumption, $\alpha_{s}+n>d_{n}$, and therefore in this case $\deg q_{\alpha_{s}}\left(k\right)<n$.
Thus if $\alpha_{s}>\alpha_{n}$ for some $s<n$, we have $\deg q_{\alpha_{n}}>\deg q_{\alpha_{s}}$.
In particular, $\alpha\geqslant\alpha_{s}>\alpha_{n}$ and therefore
$\deg q_{\alpha}<\deg q_{\alpha_{n}}$. This proves one direction.
\item In the other direction, assume that $\alpha=\max_{j}\alpha_{j}=\alpha_{n}$.
Clearly $\deg q_{\alpha}=\deg q_{\alpha_{n}}=n$, but on the other
hand it is always true that $\deg q_{\ell}\leqslant n$.
\end{enumerate}
This concludes the proof.\end{svmultproof2}
\begin{remark}
The condition \eqref{eq:poinc-sys-cond} in fact means that the point
$z=\infty$ is at most a regular singularity of the operator $\Op.$
\end{remark}
So in the remainder of the section we assume that \eqref{eq:poinc-sys-cond}
is satisfied and $n>0$. The constant part of the system \eqref{eq:moment-rec-system}
is the matrix
\[
A=\begin{bmatrix}0 & 1 & 0 & \dots & 0 & \multicolumn{4}{c}{\multirow{3}{*}{\ensuremath{\vec{0}^{d_{n}\times\tau}}}}\\
0 & 0 & 1 & \dots & 0\\
\dots\\
-\beta_{-n} & -\beta_{-n+1} & \dots &  & -\beta_{-n+d_{n}-1}\\
 &  &  &  &  & 0 & 1 & 0 & \dots & 0\\
 &  &  &  &  & 0 & 0 & 1 & \dots & 0\\
\multicolumn{5}{c}{\ensuremath{\vec{0}^{\tau\times d_{n}}}} & \dots\\
 &  &  &  &  & b_{\tau} & b_{\tau-1} & \dots &  & b_{1}
\end{bmatrix},
\]
where
\[
\beta_{-n+s}\isdef\lim_{k\to\infty}\frac{q_{-n+s}\left(k\right)}{q_{\alpha_{n}}\left(k\right)}=\frac{a_{s,n}}{a_{d_{n},n}}.
\]

\begin{proposition}
\label{prop:eig-syst-mom}The set $Z_{A}$ of the eigenvalues of the
matrix $A$ is precisely the union of the roots of $p_{n}\left(x\right)$
(i.e. the singular points of the operator $\Op$) and the jump points
$\left\{ x_{i}\right\} _{i=0}^{\np+1}$.\end{proposition}
\begin{svmultproof}
This is immediate, since $A=\begin{bmatrix}C & \vec{0}\\
\vec{0} & D
\end{bmatrix}$, where $C$ is the companion matrix of $p_{n}\left(x\right)$ and
$D$ is the companion matrix of the characteristic polynomial $\Theta_{{\cal S}}\left(z\right)$
as defined in \eqref{eq:char-poly-moments-const-part}.
\end{svmultproof}
In \cite{batbinZeros} we study the following question: \emph{how
many first moments $\left\{ m_{k}\right\} $ can vanish for a nonzero
$g\in{\cal PD}\left(\Op,\np\right)$? }In particular, we prove the
following result.
\begin{theorem}[\cite{batbinZeros}]
\label{thm:moment-vanishing}Let the operator $\Op$ be of Fuchsian
type (i.e. having only regular singular points, possibly including
$\infty$). In particular, $\Op$ satisfies the condition \eqref{eq:poinc-sys-cond}.
Let $g\in{\cal PD}\left(\Op,\np\right)$.
\begin{enumerate}
\item If there is at least one discontinuity point $\xi\in\left[a,b\right]$
of $g$ at which the operator $\Op$ is nonsingular, i.e. $p_{n}\left(\xi\right)\neq0$,
then vanishing of the first $\tau+d_{n}-n$ moments $\left\{ m_{k}\right\} _{k=0}^{\tau+d_{n}-n-1}$
of $g$ implies $g\equiv0$.
\item Otherwise, let $\Lambda\left(\Op\right)$ denote the largest positive
integer characteristic exponent of $\Op$ at the point $\infty$.
In fact, the indicial equation of $\Op$ at $\infty$ is $q_{\alpha}\left(k\right)=0$.
Then the vanishing of the first $\Lambda\left(\Op\right)+1+d_{n}-n$
moments of $g$ implies $g\equiv0$.
\end{enumerate}
\end{theorem}
Everything is now in place in order to obtain the following result.
\begin{theorem}
\label{thm:moment-taylor-dom}Let $g\in{\cal PD}\left(\Op,\np\right)$
be a not identically zero function, with $\Op$ of Fuchsian type.
Then the Stieltjes transform $S_{g}\left(z\right)$ is analytic at
the origin, and the series
\[
S_{g}\left(z\right)=\sum_{k=0}^{\infty}m_{k}z^{k}
\]
converges in a disk of radius $R$ which satisfies
\[
R\geqslant R^{*}\isdef\min\left\{ \xi^{-1}:\;\xi\in Z_{A}\right\} ,
\]
where $Z_{A}$ is given by \prettyref{prop:eig-syst-mom}. Furthermore,
for every
\[
N\geqslant\max\left\{ \tau-1,\Lambda\left(\Op\right)\right\} +d_{n}-n,
\]
$S_{g}$ satisfies $\left(N,R,S\left(k\right)\right)$ Taylor domination,
where $S\left(k\right)$ is given by \prettyref{prop:Root}. \end{theorem}
\begin{svmultproof2}
By \prettyref{lem:cond-for-poincare-syst} and \prettyref{thm:pituk-system}
it is clear that
\[
\limsup_{k\to\infty}\sqrt[k]{\left|m_{k}\right|}\leqslant\frac{1}{R^{*}}.
\]
By \prettyref{thm:moment-vanishing}, $m_{k}\neq0$ for at least some
$k=0,\dots,N$. Then we apply \prettyref{prop:Root}.
\end{svmultproof2}
In order to bound the number of zeros of $S_{g}$, we would need to
estimate the growth of the rational functions $\frac{q_{-n+s}\left(k\right)}{q_{\alpha_{n}}\left(k\right)}$,
and this can hopefully be done using some general properties of the
operator $\Op$. Then we would apply the results of \prettyref{sec:poincare-type-rec}.
We expect that in this way we can single out subclasses of ${\cal PD}$
for which uniform Taylor domination takes place. We plan to carry
out this program in a future work.
\begin{remark}
It is possible to obtain Taylor domination for the Stieltjes transforms
$S_{g}\left(z\right)$ by another method, based on Remez-type inequalities
\cite{remez1936prop,yomdin2011remez}. We plan to present these results separately.
\end{remark}
We complete this paper with one example and one question.

\smallskip

\noindent{\bf Example 2.} Consider a very special case of $D$-finite functions: polynomial $g(x)$. In this case a uniform
Taylor domination (depending only on the degree) in the maximal disk of convergence for $S_g(z)$ was obtained in
\cite{yomdin2014Bautin} via the study of the ``Bautin ideals'', generated by the moments $m_k(g)$ (considered as polynomials
in the parameters of the problem).

\smallskip

\noindent{\bf Problem 2.} {\it Moments of the form $m_k(p,Q,a,b)=\int_a^b P^k(x)q(x)dx$ play an especially important role in 
Qualitative Theory of ODEs (compare \cite{briskin2010center}). Here $P(x)$ and $q(x)$ are
univariate complex polynomials. Via change of integration variable $w=P(x)$ these moments can be reduced to the usual moments
of a piecewise-algebraic $g$, but along a certain complex curve. Does the generating function of $m_k(p,Q,a,b)$ satisfy
uniform Taylor domination (depending only on the degrees of $P$ and $q$) in its disk of convergence?}

\bibliographystyle{plain}

\begin{thebibliography}{10}

\bibitem{bat2008}
D.~Batenkov.
\newblock {Moment inversion problem for piecewise D-finite functions}.
\newblock {\em Inverse Problems}, 25(10):105001, October 2009.

\bibitem{batbinZeros}
D.~Batenkov and G.~Binyamini.
\newblock {Moment vanishing of piecewise solutions of linear ODE's}.
\newblock {\em arXiv:1302.0991.}

\bibitem{Sampl}
D.~Batenkov, N. Sarig and Y.~Yomdin.
\newblock {Accuracy of Algebraic Fourier Reconstruction for Shifts of Several Signals}.
\newblock {\em Accepted for publication in Special Issue of STSIP}, 2014.

\bibitem{bieberbach1955analytische}
L.~Bieberbach.
\newblock {\em Analytische Fortsetzung}.
\newblock Springer Berlin, 1955.

\bibitem{biernacki1936fonctions}
M.~Biernacki.
\newblock Sur les fonctions multivalentes d'ordre p.
\newblock {\em CR Acad. Sci. Paris}, 203:449--451, 1936.

\bibitem{bodine2004asymptotic}
S.~Bodine and D.A. Lutz.
\newblock Asymptotic solutions and error estimates for linear systems of
  difference and differential equations.
\newblock {\em Journal of mathematical analysis and applications},
  290(1):343--362, 2004.

\bibitem{borcea2011parametric}
J.~Borcea, S.~Friedland, and B.~Shapiro.
\newblock {Parametric Poincar{\'e}-Perron theorem with applications}.
\newblock {\em Journal d'Analyse Math{\'e}matique}, 113(1):197--225, 2011.

\bibitem{briskin2010center}
M.~Briskin, N.~Roytvarf, and Y.~Yomdin.
\newblock {Center conditions at infinity for Abel differential equations}.
\newblock {\em {Annals of Mathematics}}, 172(1):437--483, 2010.

\bibitem{Brui}
N. G. ~de Bruijn.
\newblock {On Tur{\'a}n's first main theorem}.
\newblock {\em {Acta Math. Hung.}}, 11:213--2016, 1960.

\bibitem{Car}  M. L. ~Cartwright.
\newblock {Some inequalities in the theory of functions}.
\newblock {\em {Math. Ann.}}, 111 :98--118, 1935.

\bibitem{coppel1971dichotomies}
W.~Coppel.
\newblock Dichotomies and stability theory.
\newblock In {\em Proceedings of the Symposium on Differential Equations and
  Dynamical Systems}, pages 160--162. Springer, 1971.

\bibitem{de1985proof}
L.~De~Branges.
\newblock {A proof of the Bieberbach conjecture}.
\newblock {\em Acta Mathematica}, 154(1):137--152, 1985.

\bibitem{elaydi2005ide}
S.~Elaydi.
\newblock {\em {An Introduction to Difference Equations}}.
\newblock Springer, 2005.

\bibitem{friedland2011observation}
O.~Friedland and Y.~Yomdin.
\newblock {An observation on Tur\'an-Nazarov inequality}.
\newblock {\em Studia Math}. 218, no. 1, 27--39, 2013.

\bibitem{hayman1994multivalent}
W.K. Hayman.
\newblock {\em Multivalent functions}, volume 110.
\newblock Cambridge University Press, 1994.

\bibitem{kloeden2011non}
P.~Kloeden and C.~Potzsche.
\newblock Non-autonomous difference equations and discrete dynamical systems.
\newblock {\em Journal of Difference Equations and Applications},
  17(2):129--130, 2011.
  
\bibitem{Ineq}  
D. S. Mitrinovic, J. Pecaric, A. M. Fink.
\newblock {\em Classical and New Inequalities in Analysis},
\newblock Series: Mathematics and its Applications, Vol. 61. 1993, XVIII, 740 p.

\bibitem{nazarov1994local}
F.L. Nazarov.
\newblock {Local estimates of exponential polynomials and their applications to
  inequalities of uncertainty principle type}.
\newblock {\em St Petersburg Mathematical Journal}, 5(4):663--718, 1994.

\bibitem{perron1921summengleichungen}
O.~Perron.
\newblock {{\"U}ber summengleichungen und Poincar{\'e}sche
  differenzengleichungen}.
\newblock {\em Mathematische Annalen}, 84(1):1--15, 1921.

\bibitem{pituk2002more}
M.~Pituk.
\newblock {More on Poincar{\'e}'s and Perron's Theorems for Difference
  Equations}.
\newblock {\em The Journal of Difference Equations and Applications},
  8(3):201--216, 2002.

\bibitem{poincare1885equations}
H.~Poincare.
\newblock Sur les {\'e}quations lin{\'e}aires aux diff{\'e}rentielles
  ordinaires et aux diff{\'e}rences finies.
\newblock {\em American Journal of Mathematics}, 7(3):203--258, 1885.

\bibitem{potzsche2010geometric}
C.~P{\"o}tzsche.
\newblock Geometric theory of discrete nonautonomous dynamical systems.
\newblock {\em Lecture Notes in Mathematics}, 2010.

\bibitem{remez1936prop}
E.J. Remez.
\newblock {Sur une propri{\'e}t{\'e} des polyn{\^o}mes de Tchebyscheff, Comm}.
\newblock {\em Inst. Sci. Kharkow}, 13:93--95, 1936.

\bibitem{roytwarf1997bernstein}
N.~Roytwarf and Y.~Yomdin.
\newblock Bernstein classes.
\newblock {\em Annales de l'institut Fourier}, 47:825--858, 1997.

\bibitem{turan1953neue}
P.~Tur{\'a}n.
\newblock {\em Eine neue Methode in der Analysis und deren Anwendungen}.
\newblock Akad{\'e}miai Kiad{\'o}, 1953.

\bibitem{turan1984new}
P.~Tur{\'a}n, G.~Hal{\'a}sz, and J.~Pintz.
\newblock {\em On a new method of analysis and its applications}.
\newblock Wiley-Interscience, 1984.

\bibitem{yom.nonaut.dyn}
Y.~Yomdin.
\newblock {Nonautonomous linearization}.
\newblock {\em Dynamical systems (College Park, MD)}. Lecture Notes in Math., 1342,
Springer, Berlin, 718--726, 1988.

\bibitem{yomdin2010sing.Prony}
Y.~Yomdin.
\newblock {Singularities in algebraic data acquisition}.
\newblock {\em Real and complex singularities}. London Math. Soc. Lecture Note Ser., 380,
378--396, Cambridge Univ. Press, Cambridge, 2010.


\bibitem{yomdin2011remez}
Y.~Yomdin.
\newblock Remez-type inequality for discrete sets.
\newblock {\em Israel Journal of Mathematics}, 186(1):45--60, 2011.

\bibitem{yomdin2014Bautin}
Y.~Yomdin.
\newblock {Bautin ideals and Taylor domination}.
\newblock {\em Publ. Mat.}, 58, 529--541, 2014.

\end{thebibliography}

\end{document}